\documentclass[12pt]{amsart}
\usepackage{amssymb,amsmath,amsthm,amsfonts,amsopn,url,color,hyperref,enumerate,mathtools}
\usepackage[normalem]{ulem}
\usepackage{amscd}
\usepackage[mathscr]{euscript}
\usepackage{mathtools,rotating}
\usepackage[all,2cell]{xy}
\UseAllTwocells
\input xy
\xyoption{all}
\usepackage{pgf,tikz}
\usepackage{mathrsfs}
\usetikzlibrary{arrows}
\usetikzlibrary{patterns}
\usepackage{tikz,pgfplots}
\usetikzlibrary{matrix, cd}
\usepackage{esvect}
\usepackage{tabularx}
\usepackage{arydshln}
\usepackage{array}
\newtheorem{theorem}{Theorem}[section]

\newtheorem{lemma}[theorem]{Lemma}
\newtheorem{proposition}[theorem]{Proposition}
\newtheorem{observation}[theorem]{Observation}
\newtheorem{corollary}[theorem]{Corollary}
\theoremstyle{definition}
\newtheorem*{question*}{Question}
\theoremstyle{definition}

\author{Hannah Klawa}
\address{Department of Mathematical Sciences \\ George Mason University \\ Fairfax, VA  22030}
\email{hklawa@gmu.edu}

\title{Global perinormality in a generalized $D + M$ construction}

\begin{document}

\begin{abstract} A domain $R$ is \emph{perinormal} if every going-down overring is flat and a perinormal domain $R$ is \emph{globally perinormal} if every flat overring is a localization of $R$~\cite{EpSh-peri}. I show that global perinormality is preserved in a pullback construction which encompasses a classical $D+M$ construction. In doing so, a result is given for the transfer of the property that every flat overring is a localization in the pullback construction considered.
\end{abstract}

\maketitle
\section{Introduction}

It is well-known that many classes of domains are preserved in the classical $D + M$ construction. A domain $R$ is \emph{perinormal} if every going-down overring is flat and a perinormal domain $R$ is \emph{globally perinormal} if every flat overring is a localization of $R$~\cite{EpSh-peri}. In~\cite[Theorem 2.7]{EpSh-peripull}, Epstein and Shapiro showed that perinormality is preserved in a more general version of the classical $D + M$ pullback in which one removes the condition that $V$ is of the form $k + M$ where $k$ is the residue field of $V$. In a  generalization of the classical $D + M$ construction considered in~\cite[Section 2]{Fon-toprings}, one starts with a local domain rather than a valuation domain. The purpose of this paper is to consider the preservation of global perinormality in this more general construction. Let $(T, M)$ be a local domain, $k$ the residue field of $T$, and $D$ be an integral domain with fraction field $k$. Building on the results in~\cite[Section 2]{Fon-toprings} and making use of~\cite[Theorem 2]{Ric-genqr}, we will study the transfer of the property that every flat overring is a localization to obtain that if the pullback $D \times_k T$ is globally perinormal, then $D$ and $T$ are globally perinormal. A partial converse is obtained by adding the assumption that $T$ is a valuation domain. This is analogous to the result in~\cite[Theorem 2.7(h)]{Fon-toprings} for going-down domains.

\section{Preliminary Material}

The following theorem by Richman gives a criterion for an overring of an integral domain $R$ to be flat as a $R$-module and is utilized in Section~\ref{fqrglobal}.

\begin{theorem}[\cite{Ric-genqr}, Theorem 2] Let $R$ be an integral domain and $S$ an overring of $R$. Then $S$ is flat over $R$ if and only if $S_{\mathfrak{m}} = 
R_{\mathfrak{m} \cap R}$ for every $\mathfrak{m} \in \textnormal{Max}(S)$. 
\end{theorem}

Given a diagram 
$$\xymatrix{ & S \ar@{^{(}->}[d]^-\iota\\ T \ar@{->>}[r]^-\pi & R}$$
in the category of commutative rings, let 
$$S \times_R T \coloneqq  \{t \in T \,\vert \, \pi(t) \in \iota(S)\}$$
denote the pullback of $S$ and $T$ over $R$. Note that $S \times_R T$ is a subring of $T$.

\begin{observation}\label{pullback-equal} Let $\pi: T \twoheadrightarrow R$ be a surjective ring homomorphism and $D$ and $S$ be subrings of $R$. If $S \times_R T = D \times_R T$, then $S = D$. 
\end{observation}

\begin{proof} This follows immediately because $\pi$ is surjective and $S$ and $D$ are subrings of $R$.
\end{proof}

\begin{lemma}\label{pullbacklocalization} Let $D$ be an integral domain with fraction field $k$ and let $\pi: T \twoheadrightarrow k$ be a surjective homomorphism. Let $W$ be a multiplicative subset of $D$. Then $D_W \times_k T = (D \times_k T)_V$ where $V \coloneqq \pi^{-1}(W)$.

\begin{proof} Let $W$ be a multiplicative subset of $D$ and let $V = \pi^{-1}(W)$. Let $y \in (D \times_k T)_V$. Then there is some $v \in V$ such that $vy \in D \times_k T$ so $\pi(vy) = \pi(v)\pi(y) \in D$. Because $\pi(v) \in W$, $\pi(y) = \frac{\pi(vy)}{\pi(v)} \in D_W$. Hence $y \in D_W \times_k T$.

Let $x \in D_W \times_k T$. Then $\pi(x) \in D_W$ and hence $\pi(x) = \frac{d}{w}$ for some $d \in D$, $w \in W$. Then there exists $t \in T$ and $v \in V$ such that $\pi(t) = d$ and $\pi(v) = w$. So $\pi(xv) = \pi(x)\pi(v) = \pi(t) \in D$. Thus $xv \in D \times_k T$ and hence $x = \frac{1}{v}(xv) \in (D \times_k T)_V$.
\end{proof}
\end{lemma}

The following propositions from~\cite{Fon-toprings} are essential for the work which follows and are stated here for the reader's convenience.

\begin{proposition}\label{fon-1.9} Let $D$ be an integral domain with fraction field $k$ and let $\pi: T \twoheadrightarrow k$ be a surjective homomorphism. If $W$ is a multiplicative subset of $D \times_k T$, then $(D \times_k T)_W = D_{W_D} \times_{k_{W_k}} T_W$ where $W_D$ is the image of $W$ in $D$ and $W_k$ is the image of $W$ in $k$.

\begin{proof} This is a special case of part of~\cite[Proposition 1.9]{Fon-toprings}. Equality holds with the pullback defined as above.
\end{proof}
\end{proposition}

\begin{proposition}\label{Fontana-props} Let $D$ be an integral domain with fraction field $k$ and $(T,M)$ a local domain with residue field $k$. Let $P \in \textnormal{Spec}(D \times_k T)$.
\begin{enumerate}
\item The ideal $M$ is a common ideal of $T$ and $D \times_kT$ and every $P \in \textnormal{Spec}(D \times_k T)$ is comparable with $M$ with respect to inclusion.
\item There is an isomorphism between the lattice of all the ideals of $D$ and all the ideals of $D \times_k T$ containing $M$.
\item If $P \subseteq M$ and $Q$ is the unique prime ideal of $T$ corresponding to $P$, $(D \times_k T)_P = T_Q$.
\item If $M \subseteq P$, $(D \times_k T)_P = D_{\mathfrak{p}} \times_k T$ where $\mathfrak{p}$ is the unique prime ideal of $D$ corresponding to $P$.
\end{enumerate}
\begin{proof} Because $D \times_k T$ is a subring of $T$, $M = \textnormal{ker}(T \twoheadrightarrow k)$ is a common ideal of $T$ and $D \times_k T$ by~\cite[Lemma 1.1.4(1)]{Prufer}. The other claims are~\cite[Proposition 2.1(4), 2.2(2), 2.2(3) and 2.2(6)]{Fon-toprings} respectively. Equality holds with the pullback defined as above.
\end{proof}
\end{proposition}

\section{Preservation of flat overrings and global perinormality}\label{fqrglobal}

Proposition~\ref{fqrpullback} gives a result for the transfer of the property that every flat overring is a localization. From this, one readily obtains a result for the preservation of global perinormality as a corollary. To prove Proposition~\ref{fqrpullback}, we will utilize the following analog of~\cite[Theorem 3.1]{GiBa-overringsDM} in which we observe that the valuation domain need not be of the form $k + M$ where $k$ is the residue field $V$.

\begin{proposition}[Analog of Theorem 3.1 in~\cite{GiBa-overringsDM} with $V$ not necessarily equal to $k + M$.]\label{GiBa-Thm3.1} Let $(V,M)$ be a valuation domain with residue field $k$ and fraction field $K$.  Let $D$ be an integral domain that is a subring of $k$. Then each $D \times_k V$-submodule of $K$ compares with $V$ with respect to inclusion. Furthermore, the set of overrings of $D \times_k V$ is $\{S_{\alpha}\}_{\alpha \in \mathcal{A}} \cup \{T_{\beta} \times_k V\}_{\beta \in \mathcal{B}}$ where $\{S_{\alpha}\}_{\alpha\in\mathcal{A}}$ is the set of overrings of $V$ and $\{T_{\beta}\}_{\beta \in\mathcal{B}}$ is the set of subrings of $k$ containing $D$.

\begin{proof} Let $\pi: V \twoheadrightarrow k$ be the canonical mapping.

The proof in~\cite[proof of Theorem 3.1]{GiBa-overringsDM} that each $D \times_k V$-submodule of $K$ compares with $V$ with respect to inclusion follows through without modification.

The set of overrings of $D \times_k V$ contains the set $\{S_{\alpha}\}_{\alpha \in \mathcal{A}} \cup \{T_{\beta} \times_k V\}_{\beta \in \mathcal{B}}$ where $\{S_{\alpha}\}_{\alpha\in\mathcal{A}}$ is the set of overrings of $V$ and $\{T_{\beta}\}_{\beta \in\mathcal{B}}$ is the set of subrings of $k$ containing $D$. Any overring of $D \times_k T$ compares with $V$ with respect to inclusion. Clearly any overring of $V$ is in the described set. Let $T$ be an overring of $D \times_k V$ such that $T \subseteq V$. Then $$\pi^{-1}(D) \slash M  \subseteq T \slash M \subseteq V \slash M = k$$ so $T \slash M$ is a subring of $k$ containing $\pi^{-1}(D) \slash M = D$. Note that 
\begin{eqnarray*}
T & = &  \{v \in V\,: \, v \in T\}\\
& = & \{v \in V \, : \, v + M \in T\slash M\}\\
& = & \{v \in V\, : \, \pi(v) \in T \slash M \}\\
& = & \pi^{-1}(T \slash M).
\end{eqnarray*}
\end{proof}
\end{proposition}

Before proving Proposition~\ref{fqrpullback}, we prove the following lemma using Richman's Criterion for flatness~\cite[Theorem 2]{Ric-genqr}.

\begin{lemma}\label{pullbackflat} Let $(T, M)$ a local domain with residue field $k$ and $D$ be domain with fraction field $k$. Let $S$ be an overring of $D$. Then $S$ is a flat overring of $D$ if and only if $S \times_k T$ is a flat overring of $D \times_k T$.

\begin{proof} In what follows, let $\varphi: S \times_k T \twoheadrightarrow S$ be the canonical projection. Suppose that $S$ is a flat overring of $D$. Let $\tilde{N} \in \textnormal{Max}(S \times_k T)$ be arbitrary. Then there exists a unique $N \in \textnormal{Max}(S)$ corresponding to $\tilde{N}$ by Proposition~\ref{Fontana-props} (1) and (2). 
Because $S$ is a flat overring of $D$, by~\cite[Theorem 2]{Ric-genqr}, $S_N = D_{N \cap D}$. Again note that $N \cap D \subseteq S$ and 
$$\varphi^{-1}(N \cap D) = \varphi^{-1}(N) \cap \varphi^{-1}(D) = \tilde{N} \cap (D \times_k T)$$
where $\varphi: S \times_k T \rightarrow S$ is the canonical projection. Hence $\tilde{N} \cap (D \times_k T)$ is the unique prime ideal of $D \times_k T$ corresponding to $N \cap D$. Then 
\begin{eqnarray*}
(S \times_k T)_{\tilde{N}} & = &  S_N \times_k T \:\:\text{by Proposition~\ref{Fontana-props}(4)}\\
& = &D_{N \cap D} \times_k T \:\:\text{because $S_N = D_{N \cap D}$}\\
& = & (D \times_k T)_{\tilde{N} \cap (D \times_k T)} \:\:\text{by Proposition~\ref{Fontana-props}(4).}
\end{eqnarray*}
Because $\tilde{N} \in \textnormal{Max}(S \times_k T)$ was arbitrary, $S \times_k T$ is flat over $D \times_k T$ by~\cite[Theorem 2]{Ric-genqr}.
\bigskip

Conversely, suppose that $S \times_k T$ is a flat overring of $D \times_k T$. Let $N \in \textnormal{Max}(S)$ be arbitrary. Then there is a unique $\tilde{N} \in \textnormal{Max} (S \times_k T)$ corresponding to $N$ by Proposition~\ref{Fontana-props} (1) and (2). Because $S \times_k T$ is a flat overring of $D \times_k T$, 
$$(S \times_k T)_{\tilde{N}} = (D\times_k T)_{\tilde{N} \cap (D \times_k T)}$$
by~\cite[Theorem 2]{Ric-genqr}. Note that $N \cap D \subseteq S$ and 
$$\varphi^{-1}(N \cap D) = \varphi^{-1}(N) \cap \varphi^{-1}(D) = \tilde{N} \cap (D \times_k T).$$
So $\tilde{N} \cap (D \times_k T)$ is the unique prime ideal of $D \times_k T$ corresponding to $N \cap D$.
Then 
\begin{eqnarray*}
S_N \times_k T &= &(S \times_k T)_{\tilde{N}}\:\:\text{by Proposition~\ref{Fontana-props}(4)}\\
& = &(D \times_k T)_{\tilde{N} \cap (D \times_k T)}\:\:\text{as noted above}\\
& = &D_{N \cap D} \times_k T\:\:\text{by Proposition~\ref{Fontana-props}(4)}.
\end{eqnarray*}
So $S_N = D_{N \cap D}$ by Observation~\ref{pullback-equal}. Because $N \in \textnormal{Max}(S)$ was arbitrary, $S$ is a flat overring of $D$ by~\cite[Theorem 2]{Ric-genqr}. 
\end{proof}
\end{lemma}

\begin{proposition}\label{fqrpullback} Let $D$ be an integral domain with fraction field $k$. Let $(T,M)$ be a local domain with residue field $k$. If $D \times_k T$ has the property that every flat overring is a localization of $D \times_k T$, then so does $D$. The converse holds if $T$ is a valuation domain.

\begin{proof} Let $D$ be an integral domain with fraction field $k$. Let $(T, M)$ be a local domain with residue field $k$ and let $\pi: T \twoheadrightarrow k$ be the canonical homomorphism. Suppose that every flat overring of $D \times_k T$ is a localization of $D \times_k T$. Let $S$ be a flat overring of $D$. By Lemma~\ref{pullbackflat}, $S \times_k T$ is a flat overring of $D \times_k T$. Because every flat overring of $D \times_k T$ is a localization of $D \times_k T$, $S \times_k T = (D \times_k T)_W$ where $W$ is a multiplicative subset of $D \times_k T$. By Proposition~\ref{fon-1.9}, $(D \times_k T)_W = D_{W_D} \times_{k_{W_k}} T_W$ where $W_D$ is the image of $W$ in $D$ and $W_k$ is the image of $W$ in $k$. If $W \cap \textnormal{ker}(\pi) = \emptyset$, then $W$ consists of units of $T$ and hence $S \times_k T = D_{W_D} \times_{k_W} T_W = D_{W_D} \times_k T$. Hence $S = D_{W_D}$ by Observation~\ref{pullback-equal}. If $W \cap \textnormal{ker}(\pi) \neq \emptyset$, then 
$$S \times_k T = D_{W_D} \times_{k_{W_D}} T_W = 0 \times_0 T_W = T_W.$$
Hence $S \times_k T = T$ because $S \times_k T \subseteq T$. It follows that $S = k = D_{D\backslash\{0\}}$.

Now let $(T, M)$ be a valuation domain. Suppose that every flat overring of $D$ is a localization of $D$. Let $S$ be an arbitrary flat overring of $D \times_k T$. Then $S$ is an overring of $T$ or $S = R \times_k T$ where $R$ is an overring of $D$ by Proposition~\ref{GiBa-Thm3.1}. First, suppose $S$ is an overring of $T$. Then $S = T_P$ for some prime ideal $P$ of $T$ because every overring of a valuation domain is a localization at a prime ideal (see~\cite[Theorem 10.1]{Mat}). By Proposition~\ref{Fontana-props}(3), $T$ is a localization of $D \times_k T$. Hence $S$ is a localization of $D \times_k T$. Now suppose $S = R \times_k T$ where $R$ is an overring of $D$. By Lemma~\ref{pullbackflat}, $R$ is flat over $D$. Then $R = D_W$ where $W$ is a multiplicative subset of $D$ because every flat overring of $D$ is a localization of $D$. Then $S = R \times_k T = D_W \times_k T$ is a localization of $D \times_k T$ by Lemma~\ref{pullbacklocalization}.
\end{proof}
\end{proposition}

In~\cite[Theorem 2.7]{EpSh-peripull} which is given below for the reader's convenience, Epstein and Shapiro showed that perinormality is preserved in a more general version of the classical $D + M$ pullback.

\begin{theorem}[\cite{EpSh-peripull}, Theorem 2.7] Let $D$ be an integral domain with fraction field $k$. Let $V$ be a valuation domain with residue field $k$. Then $D \times_k V$ is perinormal if and only if $D$ is perinormal.
\end{theorem}

In order to give the corollary in the generality in which it will be stated, it is necessary to note that the valuation domain assumption was not necessary in one direction.

\begin{observation}\label{genD+Mperi} Let $D$ be an integral domain with fraction field $k$. Let $(T, M)$ be a local domain with residue field $k$. If the pullback $D \times_k T$ is perinormal, then $D$ and $T$ are perinormal.

\begin{proof} Let $D$ be an integral domain with fraction field $k$. Let $(T, M)$ be a local domain with residue field $k$. By Proposition~\ref{Fontana-props}(3), $T$ is a localization of $D \times_k T$ which is perinormal. Thus $T$ is perinormal by~\cite[Proposition 2.5]{EpSh-peri}.

The perinormality of $D$ follows from the proof of~\cite[Theorem 2.7]{EpSh-peripull}.

\end{proof}
\end{observation}

As a corollary, we obtain the following analog of~\cite[Theorem 2.7]{EpSh-peri} for globally perinormal domains.

\begin{corollary}\label{gp-D+M} Let $D$ be an integral domain with fraction field $k$. Let $(T, M)$ be a local domain with residue field $k$. If the pullback $D \times_k T$ is globally perinormal, then $D$ and $T$ are globally perinormal. If $T$ is a valuation domain and $D$ is globally perinormal, then $D \times_k T$ is globally perinormal.
\end{corollary}

\begin{proof} Let $D$ be an integral domain with fraction field $k$. Let $(T, M)$ be a local domain with residue field $k$. Suppose $D \times_k T$ is globally perinormal. So every flat overring of $D \times_k T$ is a localization of $D \times_k T$. By Observation~\ref{genD+Mperi}, $D$ is perinormal. To show that $D$ is globally perinormal, it suffices to show that every flat overring of $D$ is a localization of $D$ but this follows directly by Proposition~\ref{fqrpullback}.

By Proposition~\ref{Fontana-props}(3), $T$ is a localization of $D \times_k T$ which is globally perinormal. Hence $T$ is globally perinormal by~\cite[Proposition 6.1]{EpSh-peri}.

Now suppose $T$ is a valuation domain and $D$ is globally perinormal. Then $D \times_k T$ is perinormal by~\cite[Theorem 2.7]{EpSh-peripull}. Because every flat overring of $D$ is a localization of $D$, it follows by Proposition~\ref{fqrpullback} that every flat overring of $D \times_k T$ is a localization of $D \times_k T$. Hence $D \times_k V$ is globally perinormal. 
\end{proof}

\section*{Acknowledgments}
I am very grateful to my advisor Neil Epstein for suggesting the question which led to this work and for helpful advice and input along the way. I am also grateful to Evan Houston for catching an error in a previous version of the paper and the referee for helpful comments which improved the paper.

\bibliographystyle{amsalpha}
\bibliography{references}

\providecommand{\bysame}{\leavevmode\hbox to3em{\hrulefill}\thinspace}
\providecommand{\MR}{\relax\ifhmode\unskip\space\fi MR }
\providecommand{\MRhref}[2]{%
  \href{http://www.ams.org/mathscinet-getitem?mr=#1}{#2}
}
\providecommand{\href}[2]{#2}
\begin{thebibliography}{FHP97}

\bibitem[BG73]{GiBa-overringsDM}
Eduardo Bastida and Robert Gilmer, \emph{Overrings and divisorial ideals of
  rings of the form ${D} + {M}$}, The Michigan Mathematics Journal \textbf{20}
  (1973), 79--95.

\bibitem[ES16]{EpSh-peri}
Neil Epstein and Jay Shapiro, \emph{Perinormality - a generalization of {K}rull
  domains}, Journal of Algebra \textbf{451} (2016), 65--84.

\bibitem[ES19]{EpSh-peripull}
\bysame, \emph{Perinormality in pullbacks}, Journal of Commutative Algebra
  \textbf{11} (2019), no.~3, 341–362.

\bibitem[FHP97]{Prufer}
Marco Fontana, James~A. Huckaba, and Ira~J. Papick, \emph{Pr\"{u}fer domains},
  Marcel Dekker, Inc., 1997.

\bibitem[Fon80]{Fon-toprings}
Marco Fontana, \emph{Topologically defined classes of commutative rings},
  Annali di Matematica Pura ed Applicata \textbf{123} (1980), 331--355.

\bibitem[Mat86]{Mat}
Hideyuki Matsumura, \emph{Commutative ring theory}, Cambridge University Press,
  1986.

\bibitem[Ric65]{Ric-genqr}
Fred Richman, \emph{Generalized quotient rings}, Proceedings of the American
  Mathematical Society \textbf{16} (1965), 794–799.

\end{thebibliography}

\end{document}